\documentclass[makeidx,a4paper,11pt]{article}
\usepackage{makeidx}
\usepackage{amssymb , amsmath, amsthm}
\newtheorem{defi}{Definition} 
\newtheorem{thm}[defi]{Theorem}

\newtheorem{rem}[defi]{Remark}
 \newtheorem{prop}[defi]{Proposition}
\newtheorem{lemme}[defi]{Lemma}
\newtheorem{cor}[defi]{Corollary}
 
\newcommand{\twosystem}[2]{\left\{\begin{aligned} &#1\\ &#2\end{aligned}\right.}
\newcommand{\threesystem}[3]{\left\{ \begin{aligned}&#1\\ &#2\\&#3\end{aligned}\right.}

\newcommand{\nero}{\smallskip$\bullet\quad$\rm}

\newcommand{\twoarray}[2]{\begin{aligned}&{#1}\\&{#2}\end{aligned}}

\newcommand{\starred}[1]{#1^{\star}}
\newcommand{\isoper}{\dfrac{{\rm Vol}(\Sigma)}{{\rm Vol}(\Omega)}}

\newcommand{\scal}[2]{\langle{#1},{#2}\rangle}

\newcommand{\Cal }[1]{{\mathcal {#1}}}
\newcommand{\bd}{\partial}
\newcommand{\abs}[1]{\lvert{#1}\rvert}

\newcommand{\reals}{{\bf R}}
\newcommand{\real}[1]{{\bf R}^{#1}}
\newcommand{\sphere}[1]{{\bf S}^{#1}}
\newcommand{\ball}[1]{{\bf B}^{#1}}

\newcommand{\Vol}{{\rm Vol}}

\newcommand{\prin}{\dfrac{\theta'}{\theta}}

\begin{document}

\title{On the spectrum of the Dirichlet-to-Neumann operator acting on forms of a Euclidean domain\footnote{Classification AMS $2000$: 58J50, 35P15 \newline 
Keywords: Manifolds with boundary, Differential forms, Dirichlet-to-Neumann operator, Eigenvalues}} 
\author{ S. Raulot and A. Savo}
\date{\today}
 
\maketitle

\begin{abstract}
We compute the whole spectrum of the Dirichlet-to-Neumann operator acting on differential $p$-forms on the unit Euclidean ball. Then, we prove a new upper bound for its first eigenvalue on  a domain $\Omega$ in Euclidean space in terms of the isoperimetric ratio ${\rm Vol}(\bd\Omega)/{\rm Vol}(\Omega)$.
\end{abstract}


\section{Introduction}


Let $(\Omega^{n+1},g)$ be an $(n+1)$-dimensional compact and connected Riemannian manifold with smooth boundary $\Sigma$. The Dirichlet-to-Neumann operator on functions associates, to each function defined on the boundary, the normal derivative of
its harmonic extension to $\Omega$. More precisely, if $f\in C^{\infty}(\Sigma)$, its harmonic extension $\widehat{f}$ is the unique smooth function on $\Omega$ satisfying
$$
\twosystem
{\Delta \widehat{f}=0\text{ in }\Omega,}
{\widehat{f}= f\text{ on }\Sigma}
$$
and the Dirichlet-to-Neumann operator $T^{[0]}$ is defined by:
\begin{eqnarray*}
T^{[0]}f:=-\frac{\partial \widehat{f}}{\partial N}
\end{eqnarray*}

where $N$ is the inner unit normal to $\Sigma$. It is a well known result (see \cite{taylor} for example) that $T^{[0]}$ is a first order elliptic, non-negative and self-adjoint pseudo-differential operator with discrete spectrum  
\begin{eqnarray*}
0=\nu_{1,0}(\Omega)<\nu_{2,0}(\Omega)\leq\nu_{3,0}(\Omega)\leq\cdots\nearrow\infty. 
\end{eqnarray*} 

As $\Omega$ is connected,  $\nu_{1,0}(\Omega)=0$ is simple, and its eigenspace consists of the constant functions. 
The first positive eigenvalue has the following variational characterization:
\begin{equation}\label{car02}
\nu_{2,0}(\Omega)=\inf\Big\{\dfrac{\int_{\Omega}\abs{df}^2}{\int_{\Sigma}f^2}: f\in C^{\infty}(\Omega)\setminus\{0\}, \int_{\Sigma}f=0\Big\}.
\end{equation}

The study of the spectrum of $T^{[0]}$ was initiated by Steklov in \cite{steklov}. We note that the Dirichlet-to-Neumann map is closely related to the problem of determining a complete Riemannian manifold with boundary from the Cauchy data of harmonic functions. Indeed, a striking result of Lassas, Taylor and Uhlmann \cite{ltu} states that if the manifold $\Omega$ is real analytic and has dimension at least $3$, then the knowledge of $T^{[0]}$ determines $\Omega$ up to isometry. 
\smallskip

It can be easily seen that the eigenvalues of the Dirichlet-to-Neumann map of the unit ball $\ball{n+1}$ in $\real{n+1}$ are $\nu_{k,0}=k$, with $k=0,1,2,...$ and the corresponding eigenspace is given by the vector space of homogeneous harmonic polynomials of degree $k$ restricted to the sphere $\partial\ball{n+1}$.


\subsection{The Dirichlet-to-Neumann operator on forms}


In \cite{R-S}, we extend the definition of the Dirichlet-to-Neumann map $T^{[0]}$ acting on functions to an operator $T^{[p]}$ acting on $\Lambda^p(\Sigma)$, the vector bundle of differential $p$-forms of $\Sigma=\bd\Omega$ for $0\leq p\leq n$. This is done as follows.  Let $\omega$ be a form of degree $p$ on $\Sigma$, with $p=0,1,\dots,n$. Then there exists a unique $p$-form $\widehat{\omega}$ on $\Omega$ such that:
$$
\twosystem
{\Delta\widehat{\omega}=0}
{\starred J\widehat{\omega}=\omega,\,\,i_N\widehat{\omega}=0.}
$$
Here $\Delta=d\delta+\delta d$ is the Hodge Laplacian acting on $\Lambda^p(\Omega)$ (the bundle of $p$-forms on $\Omega$) $\starred J:\Lambda^p(\Omega)\to\Lambda^p(\Sigma)$ is the restriction map and $i_N$ is the interior product of $\widehat{\omega}$ with the inner unit normal vector field $N$. The existence and uniqueness of the form $\widehat{\omega}$ (called the {\it harmonic tangential extension of $\omega$}) is proved, for example, in Schwarz \cite{schwarz}. We let:
$$
T^{[p]}\omega=-i_Nd\widehat{\omega}.
$$
Then $T^{[p]}:\Lambda^p(\Sigma)\to \Lambda^p(\Sigma)$ defines a linear operator, {\it the (absolute) Dirichlet-to-Neumann operator}, which reduces to the classical Dirichlet-to-Neumann operator $T^{[0]}$ acting on functions when $p=0$. We proved in \cite{R-S} that $T^{[p]}$ is an elliptic self-adjoint and non-negative pseudo-differential operator, with discrete spectrum
$$
0\leq\nu_{1,p}(\Omega)\leq\nu_{2,p}(\Omega)\leq\dots
$$
tending to infinity. Note that $\nu_{1,p}(\Omega)$ can in fact be zero: it is not difficult to prove that ${\rm Ker}T^{[p]}$ is isomorphic 
to $H^p(\Omega)$, the $p$-th absolute de Rham cohomology space of $\Omega$ with real coefficients. 

\medskip

The operator $T^{[p]}$ belongs to a family of operators first considered by G. Carron in \cite{carron}. Other Dirichlet-to-Neumann operators acting on differential forms, but different from ours, were introduced by Joshi and Lionheart in \cite{joshi}, and Belishev and Sharafutdinov in \cite{belishev}. In fact, our operator $T^{[p]}$ appears in a certain matrix decomposition of the Joshi and Lionheart operator (see \cite{R-S} for complete references). However, one advantage of our operator is its self-adjointness, which permits to study its spectral and variational properties. In particular one has (see \cite{R-S}):
\begin{equation}\label{var}
\nu_{1,p}(\Omega)=\inf\Big\{\dfrac{\int_{\Omega}\abs{d\omega}^2+\abs{\delta\omega}^2}{\int_{\Sigma}\abs{\omega}^2}: \omega\in\Lambda^{p}(\Omega)\setminus \{0\}, \, i_N\omega=0
\,\,\text{on $\Sigma$}\Big\}.
\end{equation}

For $p=0,\dots, n$, we also have a dual operator $T^{[p]}_D:\Lambda^p(\Omega)\to\Lambda^p(\Omega)$ with eigenvalues
$\nu_{k,p}^D(\Omega)=\nu_{k,n-p}(\Omega)$ (for its definition, we refer to \cite{R-S}). Here we just want to observe that:
\begin{equation}\label{vardual}
\nu_{1,p}^D(\Omega)=\inf\Big\{\dfrac{\int_{\Omega}\abs{d\omega}^2+\abs{\delta\omega}^2}{\int_{\Sigma}\abs{\omega}^2}: \omega\in\Lambda^{p+1}(\Omega)\setminus \{0\}, \, \starred J\omega=0
\,\,\text{on $\Sigma$}\Big\}.
\end{equation}

In \cite{R-S}, we obtained sharp upper and lower bounds of $\nu_{1,p}(\Omega)$ in terms of the extrinsic geometry of its boundary: let us briefly explain the main lower bound. 

\smallskip

Fix $x\in\Sigma$ and consider the principal curvatures $\eta_1(x),\dots,\eta_n(x)$ of $\Sigma$ at $x$; if $p=1,\dots,n$ and $1\leq j_1<\dots<j_p\leq n$ is a multi-index, we call the number $\eta_{j_1}(x)+\dots+\eta_{j_p}(x)$ a \it p-curvature \rm of $\Sigma$. We set:
$$
\twoarray
{\sigma_p(x)=\inf\{\eta_{j_1}(x)+\dots+\eta_{j_p}(x): 1\leq j_1<\dots<j_p\leq n\}}
{\sigma_p(\Sigma)=\inf\{\sigma_p(x): x\in\Sigma\}}
$$
and say that $\Sigma$ is {\it $p$-convex} if $\sigma_p(\Sigma)\geq 0$ that is, if all $p$-curvatures of $\Sigma$ are non-negative. For example $\Sigma$ is $1$-convex  if and only if it is convex in the usual sense,  and it is $n$-convex if and only if it is mean-convex (that is, it has non-negative mean curvature everywhere). 

\smallskip

We then proved that for a compact domain $\Omega$ in $\real{n+1}$ with $p$-convex boundary, one has 
\begin{eqnarray}\label{nonsharp}
\nu_{1,p}(\Omega)> \dfrac{n-p+2}{n-p+1}\sigma_p(\Sigma)
\end{eqnarray}
for $0\leq p<\frac{n+1}{2}$ and 
\begin{eqnarray}\label{sharp}
\nu_{1,p}(\Omega)\geq \frac{p+1}p\sigma_p(\Sigma)
\end{eqnarray}
for $(n+1)/2\leq p\leq n$. The inequality \eqref{nonsharp} is never sharp, but \eqref{sharp} is sharp for Euclidean balls and actually  equality characterizes the ball when $p>(n+1)/2$.  
For all this, and for similar inequalities  in Riemannian manifolds we refer to \cite{R-S}.

\smallskip

In this paper we continue the study of the spectral properties of $T^{[p]}$. Namely:

\smallskip

\nero  We compute the whole spectrum of $T^{[p]}$ and describe its eigenforms on the unit ball in $\real{n+1}$.

\smallskip

\nero We give a sharp upper bound for the first eigenvalue of $T^{[p]}$ on Euclidean domains, in terms of the isoperimetric ratio $\isoper$.

\smallskip

It is perhaps worth noticing that in dimension $3$ we have the following interpretation of \eqref{var} and \eqref{vardual} in terms of vector fields. If $\Omega$ is a bounded domain in $\real{3}$, then for all vector fields $X$ on $\Omega$ which are tangent to the boundary $\Sigma$ one has:
\begin{equation}\label{dct}
\int_{\Omega}\Big(\abs{{\rm div}X}^2+\abs{{\rm curl}X}^2\Big)\geq \nu_{1,1}(\Omega)\int_{\Sigma}\abs {X}^2,
\end{equation}
with equality  iff $X$ is harmonic and its dual $1$-form  restricts to an eigenform of $T^{[1]}$ associated to $\nu_{1,1}(\Omega)$. Recall that,
on a three dimensional Riemannian manifold,  the curl of a vector field $X$ is the vector field defined  by 
$$
{\rm curl} X=\big(\star dX^\sharp\big)^\sharp
$$ 
where $^\sharp$ denotes the musical isomorphism between the tangent space and the cotangent space. Combining (\ref{dct}) with the estimate (\ref{nonsharp}) gives, for all Euclidean domains with convex boundary:
\begin{equation}\label{as}
\int_{\Omega}\Big(\abs{{\rm div}X}^2+\abs{{\rm curl}X}^2\Big)>\dfrac 32 \sigma_1(\Sigma)\int_{\Sigma}\abs {X}^2,
\end{equation}
where $X$ is any vector field tangent to $\Sigma$ and $\sigma_1(\Omega)$ is a lower bound of the principal curvatures of $\Sigma$. As a by-product of the calculation in Section \ref{SB}, we will see that \eqref{as} is almost sharp, because for all vector fields on the unit ball (tangent to the boundary)
we have the sharp inequality
\begin{equation}\label{mintan}
\int_{\ball{3}}\Big(\abs{\rm{div} X}^2 +\abs{\rm{curl} X}^2\Big)\geq \dfrac 53\int_{\bd\ball 3}\abs{X}^2
\end{equation}
(for the description of the minimizing vector fields for the inequality \eqref{mintan} we refer to Section \ref{vectors}). 

Similarly, let $X$ be a  vector field on a Euclidean domain $\Omega$ which is normal to the boundary. Then:
\begin{equation}\label{dcn}
\int_{\Omega}\Big(\abs{{\rm div}X}^2+\abs{{\rm curl}X}^2\Big)\geq \nu_{1,2}(\Omega)\int_{\Sigma}\abs {X}^2,
\end{equation}
with equality iff  $X$ is harmonic and the Hodge-star of its dual $1$-form restricts to an eigenform of $T^{[2]}$ associated to $\nu_{1,2}(\Omega)$. Using (\ref{sharp}), we see that, if $\Sigma$ is mean-convex:
\begin{equation}\label{mc}
\int_{\Omega}\Big(\abs{{\rm div}X}^2+\abs{{\rm curl}X}^2\Big)\geq \dfrac{3}2 \sigma_2(\Sigma)\int_{\Sigma}\abs {X}^2.
\end{equation}
Note that $\sigma_2(\Sigma)=2H$, where $H$ is a lower bound of the mean curvature of $\Sigma$. In this situation, our inequality is sharp and is an equality if and only if $\Omega$ is a ball in  $\real 3$, in which case the lower bound is $3$, and $X$ is a multiple of the radial vector field $r\frac{\bd}{\bd r}$
(see Section \ref{vectors}). 
  
  \smallskip
  
We end this discussion by remarking that inequalities (\ref{as}) and (\ref{mc}) continue to be true for all bounded domains in three-dimensional Riemannian manifolds with non-negative Ricci curvature.


\subsection {The spectrum of $T^{[p]}$ on the unit Euclidean ball in $\real{n+1}$}\label{SB}


In this section we compute the spectrum of $T^{[p]}$ on the unit ball $\mathbf{B}^{n+1}$ in $\real{n+1}$.

\smallskip

Let $\bar{\Delta}$ (resp. $\Delta$) denote the Hodge Laplacian acting on $p$-forms of $\sphere n$ (resp. $\real{n+1}$). It will turn out that $\bar{\Delta}$ and $T^{[p]}$ have the same eigenspaces: so we describe in details the eigenspaces of $\bar{\Delta}$. 

\medskip

We start from the case $p=0$, which is classical. The operator $\bar{\Delta}$ is simply the Laplacian on functions, and it is a well-known fact that its eigenfunctions are restrictions to $\sphere n$ of homogeneous polynomial harmonic functions on $\real{n+1}$. Precisely, let $P_{k,0}$ be the vector space of all polynomial functions on $\real{n+1}$ of homogeneous degree $k$, where $k=0,1,2,\dots$, and set:
$$
H_{k,0}=\{f\in P_{k,0}: \Delta f=0\}.
$$
Then the spectrum of $\bar{\Delta}$ acting on functions of $\sphere n$ is given by the eigenvalues
$$
\mu''_{k,0}=k(n+k-1), \quad k=0,1,2,\dots,
$$
with multiplicity $M_{k,0}={\rm dim}(H_{k,0})$ and associated eigenspace $\starred J(H_{k,0})$. 

\smallskip

Now fix an eigenfunction $f\in \starred J(H_{k,0})$ so that, by assumption, its harmonic extension $\hat f$ is a harmonic polynomial of homogeneous degree $k$. It is very easy to see that $T^{[0]}f=kf$: so, $f$ is also a Dirichlet-to-Neumann eigenfunction associated to the eigenvalue $k$. A standard density argument shows that these are all possible eigenvalues of $T^{[0]}$. Therefore we have the following well-known result:
\begin{thm}\label{DNF} 
The spectrum of the Dirichlet-to-Neumann operator $T^{[0]}$ on the unit ball in $\real{n+1}$ consists of the eigenvalues $\nu''_{k,0}=k$, where $k=0,1,2,\dots$, each with multiplicity $M_{k,0}={\rm dim}(H_{k,0})$ and associated eigenspace $\starred J(H_{k,0})$.
\end{thm}

Now assume $p=1,\dots,n$. The calculation of the spectrum of the Hodge Laplacian on the sphere was first done in \cite{gallot-meyer}. We follow the exposition in \cite{it}. 

\smallskip

As the Hodge Laplacian commutes with both the differential and the codifferential, it preserves closed (resp. co-closed) $p$-forms. Moreover, any exact eigenform is the differential of a co-exact eigenform associated to the same eigenvalue. In the following, we denote by $\{\mu'_{k,p}\}$ (resp. $\{\mu''_{k,p}\}$) the spectrum of the Hodge Laplacian restricted to closed (resp. co-closed) $p$-forms on the sphere $\sphere{n}$.
\smallskip

We can write a $p$-form on $\real{n+1}$ as:
$$
\omega=\sum_{i_1<\dots<i_p}\omega_{i_1\dots i_p}dx_{i_1}\wedge\dots\wedge dx_{i_p}
$$
and we say that $\omega$ is polynomial if each component $\omega_{i_1\dots i_p}$ is a polynomial function. 

\smallskip

Now let $P_{k,p} $ be the vector space of  polynomial $p$-forms of homogeneous degree $k\geq 0$ on $\real{n+1}$ and set:
$$
\threesystem
{H_{k,p}=\{\omega\in P_{k,p}: \Delta\omega=0, \delta\omega=0\},}
{H'_{k,p}=\{\omega\in H_{k,p}: d\omega=0\}}
{H''_{k,p}=\{\omega\in H_{k,p}: i_Z\omega=0\}}
$$
where $Z$ is the radial vector field $Z=r\frac{\partial}{\partial r}=\sum_{j=1}^{n+1}x_j\frac{\partial}{\partial x_j}$. On the boundary, $Z=-N$.
It turns out that $H_{k,p}=H'_{k,p}\oplus H''_{k,p}$ and $d:H''_{k,p}\to H'_{k-1,p+1}$ is a linear isomorphism for all $k\geq 1$. We set
$$
M_{k,p}={\rm dim}(H''_{k,p})
$$
(this number can be computed by representation theory, see Theorem 6.8 in \cite{it}).

\smallskip

By the Hodge-de Rham decomposition, any Hodge-Laplace eigenspace splits into the direct sum of its co-exact, exact and harmonic parts. But in the range $1\leq p\leq n$ the only harmonic forms occur in degree $p=n$, and are multiples of the volume form of $\sphere n$. Moreover, for $p=n$ the co-exact part is reduced to zero. Then, there is a spectral resolution of $L^2(\Lambda^p(\sphere{n}))$ which consists of the following three types of Hodge-Laplace eigenforms: $1\leq p\leq n-1$ and the eigenform $\xi$ is co-exact;
 $1\leq p\leq n$ and the eigenform $\xi$ is exact;  $p=n$ and $\xi$ is a multiple of the volume form of $\sphere n$. Correspondingly, we have the following three families of eigenvalues (see \cite{it}):
 
\nero If $1\leq p\leq n-1$ and $\xi$ is co-exact the associated eigenvalues are given by the family
$$
\mu''_{k,p}=(k+p)(n+k-p-1)\quad k=1,2,\dots
$$
with associated eigenspace $\starred J(H''_{k,p})$ and multiplicity $M_{k,p}$.

\nero If $1\leq p\leq n$ and $\xi$ is exact the associated eigenvalues are given by the family
$$
\mu'_{k,p}=(k+p-1)(n+k-p)\quad k=1,2,\dots
$$
with associated eigenspace $\starred J(H'_{k-1,p})$ and multiplicity $M_{k,p-1}$.

\nero If $p=n$ and $\xi$ is the volume form of $\sphere n$ we have the associated eigenvalue $\mu''_{1,n}=0$. 

\medskip

In conclusion, eigenforms of the Hodge Laplacian are suitable restrictions to $\sphere{n}$ of harmonic, co-closed polynomial forms. 
 
\smallskip
 
In Section \ref{HTE} we will prove that any Hodge-Laplace eigenform of $\sphere n$ is also a Dirichlet-to-Neumann eigenform (associated to a different eigenvalue). This will imply the following calculations. 

\begin{thm}\label{DNS} 
Let $1\leq p\leq n-1$. The spectrum of the Dirichlet-to-Neumann operator $T^{[p]}$ on the unit ball of $\real{n+1}$  is given by the following two families of eigenvalues $\{\nu_{k,p}''\}, \{\nu_{k,p}'\}$ indexed by a positive integer $k=1,2,\dots$:
$$
\twosystem
{\nu_{k,p}'' = k+p\quad\text{with multiplicity $M_{k,p}$},}
{\nu_{k,p}' = (k+p-1)\dfrac{n+2k+1}{n+2k-1}\qquad\text{with multiplicity $M_{k,p-1}$}.}
$$
The eigenspace associated to $\nu_{k,p}''$ is $\starred J(H''_{k,p})$ and consists of co-exact forms. 
The eigenspace associated to $\nu_{k,p}'$ is $\starred J(H'_{k-1,p})$ and consists of exact forms.
\end{thm}

In degree $p=n$:
\begin{thm}\label{DNSn} 
The spectrum of the Dirichlet-to-Neumann operator $T^{[n]}$ on the unit ball of $\real{n+1}$ consists of the eigenvalues:
$$
\nu_{1,n}''=n+1,
$$
with multiplicity one (the associated eigenspace being spanned by the volume form of $\sphere n$) and, for $k\geq 1$:
$$
\nu_{k,n}'=(k+n-1)\dfrac{n+2k+1}{n+2k-1},
$$
with multiplicity $M_{k,n-1}$ and associated eigenspace $\starred J(H'_{k-1,n}$). 
\end{thm}

From this result we deduce that the lowest eigenvalue of $T^{[n]}$ is $\nu_{1,n}=n+1$ with multiplicity one for $n\geq 2$ and three for $n=1$. Indeed, in this last situation, we have $\nu_{1,1}''=\nu_{2,1}'=2$ and the corresponding eigenspace is spanned by the the volume form $v$ of $\sphere{1}$, $\bar{d}x$ and $\bar{d}y$ where $\bar{d}$ denotes the differential on $\sphere{1}$. From the above results we obtain:
\begin{cor}\label{first}
If  $\nu_{1,p}$ denotes the first eigenvalue of $T^{[p]}$ on the unit ball in $\real{n+1}$, then:
$$
\nu_{1,p}=\left\{\aligned &\dfrac{n+3}{n+1} p\qquad\text{if}\qquad 1\leq p\leq\dfrac{n+1}2\\
& p+1\qquad\text{if}\qquad \dfrac{n+1}2\leq p\leq n.
\endaligned\right.
$$
\end{cor}

The proof of Theorems \ref{DNS} and \ref{DNSn} splits into two parts. In a first step (see Lemma \ref{PolarEu} and \ref{pol}), we compute the expression of the Hodge Laplacian on $p$-forms for rotationally symmetric manifolds. Then in Section \ref{spectre} we apply these computations to construct the tangential harmonic extension to the unit ball of any $p$-eigenform of the Hodge Laplacian on $\sphere n=\bd \ball{n+1}$.


\subsection{A sharp upper bound by the isoperimetric ratio}


As shown in \cite{R-S}, the existence of a parallel $p$-form implies upper bounds of the Dirichlet-to-Neumann eigenvalues by the isoperimetric ratio $\frac{{\rm Vol}(\Sigma)}{{\rm Vol}(\Omega)}$. These bounds are never sharp, unless $p=n$. In that case one has, for any $(n+1)$-dimensional Riemannian domain $\Omega$:
\begin{equation}\label{isoper}
\nu_{1,n}(\Omega)\leq\isoper,
\end{equation}
which is sharp for Euclidean balls. 
The proof of \eqref{isoper} is easily obtained by applying the min-max principle (\ref{var}) to the test $n$-form $\omega=\star dE$, where $E$ is the mean-exit time function, solution of the problem
\begin{equation}\label{met}
\twosystem{\Delta E=1\quad\text{on $\Omega$},}
{E=0\quad\text{on $\Sigma$}.}
\end{equation}
If $\Omega$ is a Euclidean domain and equality holds in \eqref{isoper}, then a famous result of Serrin implies that $\Omega$ is a ball. The equality case for general Riemannian domains is an open (and interesting) problem. 

\smallskip

In this paper  we generalize inequality \eqref{isoper} to any degree $p=0,\dots,n$ when $\Omega$ is a Euclidean domain; in the range $p\geq (n+1)/2$ the estimate is sharp and it also turns out that the ball is the unique maximizer.
Namely, we prove:
\begin{thm}\label{newupper}
Let $\Omega$ be a domain in $\real{n+1}$ and $p=1,\dots,n$. Then:
\begin{eqnarray*}
\nu_{1,p}(\Omega)\leq\dfrac{p+1}{n+1}\frac{\Vol(\Sigma)}{\Vol(\Omega)}.
\end{eqnarray*}
Equality holds iff $p\geq (n+1)/2$ and $\Omega$ is a Euclidean ball.
\end{thm}

We note that the corresponding inequality for the first {\it positive} eigenvalue of the Dirichlet-to-Neumann operator on functions:
$$
\nu_{2,0}(\Omega)\leq\dfrac{1}{n+1}\frac{\Vol(\Sigma)}{\Vol(\Omega)},
$$
has been recently proved by Ilias and Makhoul in \cite{ilias}, but their approach does not extend to higher degrees.


\section{The Dirichlet-to-Neumann spectrum of the unit Euclidean ball}\label{spectre} 


We first give an expression of  the Hodge Laplacian on the unit ball $\ball{n+1}$ in 
$\real{n+1}$. Note that $\ball{n+1}=[0,1]\times\sphere n$ with the metric $dr^2\oplus r^2 ds^2_n$, where $ds^2_n$ is the canonical metric on $\sphere n$.

\medskip

We consider $p$-forms on the ball $\ball{n+1}$ of the following type:
$$
\omega{(r,x)}=Q(r)\xi(x)+P(r)dr\wedge\eta(x),
$$
where $\eta\in\Lambda^{p-1}(\sphere n), \xi\in\Lambda^p(\sphere n)$, and $P, Q$ are smooth functions on $(0,1)$. We will write for short
\begin{equation}\label{type}
\omega = Q\xi+Pdr\wedge\eta.
\end{equation}

Then we have:
\begin{lemme}\label{PolarEu}  
Let $\bar{d}$ and $\bar{\delta}$ denote, respectively, the differential and the co-differential acting on $\sphere n$. Let $\omega$ be a $p$-form as in \eqref{type}, then:
$$
\Delta\omega=\omega_1+dr\wedge \omega_2,
$$
where:
$$
\begin{aligned}
\omega_1&=\dfrac {Q}{r^2}\bar\Delta\xi-\left(Q''+\dfrac{n-2p}r Q'\right)\xi-\dfrac{2P}{r}\bar d\eta,\\
\omega_2&=\dfrac{P}{r^2}\bar\Delta\eta-\left(P'+\dfrac{n-2p+2}r P\right)'\eta-\dfrac{2Q}{r^3}\bar\delta\xi.
\end{aligned}
$$
\end{lemme}

For the proof, we refer to Lemma \ref{pol} in Section \ref{polar}.


\subsection{Proof of Theorems \ref{DNS} and \ref{DNSn}}\label{HTE}


It is enough to show that any Hodge-Laplace eigenform of $\sphere n$ is also a Dirichlet-to-Neumann eigenform, associated to one of the eigenvalues $\nu'_{k,p}$ or $\nu''_{k,p}$ of Theorems \ref{DNS} and \ref{DNSn}. In fact, as the direct sum of all the Hodge-Laplace eigenspaces is dense in $L^2(\Lambda^p(\sphere n))$, the list of Dirichlet-to-Neumann eigenvalues we have just found is complete, and the theorems follow. 

\smallskip

We now use the classification of the eigenspaces of the Hodge Laplacian done in Section \ref{SB}.  It is then enough to prove the following:
\begin{prop}\label{coclosed} 
\item a) Assume $1\leq p\leq n-1$, and let $\xi$ be a co-exact Hodge $p$-eigenform on $\sphere n$ associated to 
$$\mu''_{k,p}=(k+p)(n+k-p-1)$$ 
for some $k\geq 1$. Then $\xi$ is also a Dirichlet-to-Neumann eigenform associated to the eigenvalue
$$
\nu''_{k,p}=k+p.
$$

\item b) Assume $p=n$ and let $\xi$ be a multiple of the volume form of $\sphere n$. Then $\xi$ is also a Dirichlet-to-Neumann eigenform associated to the eigenvalue $\nu''_{1,n}=n+1$.

\item c) Assume $1\leq p\leq n$ and let $\xi$ be an exact Hodge-Laplace $p$-eigenform associated to
$$
\mu'_{k,p}=(k+p-1)(n+k-p)
$$ 
for some $k\geq 1$. Then $\xi$ is also a Dirichlet-to-Neumann $p$-eigenform associated to
$$
\nu'_{k,p}=(k+p-1)\dfrac{n+2k+1}{n+2k-1}.
$$
\end{prop}

For the proof, we explicitly determine in all three cases the tangential harmonic extension of $\xi$ using Lemma \ref{PolarEu}.  

\smallskip

{\bf Proof of a)}
Let us compute the tangential harmonic extension $\widehat \xi$ of $\xi$. We let $\widehat\xi=Q\xi$ with $Q=Q(r)$ to be determined so that $\Delta\widehat\xi=0$. Note that $i_N\widehat\xi=0$; moreover $\starred J\widehat\xi=\xi$ whenever $Q(1)=1$. As $\eta=0=\bar\delta\xi$, Lemma \ref{PolarEu} gives
$$
\Delta (Q\xi)=\left(\mu''_{k,p}\dfrac{ Q}{r^2}-Q''-\dfrac{n-2p}rQ'\right)\xi
$$
hence $Q\xi$ is harmonic if $Q(r)$ satisfies:
\begin{eqnarray}\label{harm1}
r^2Q''+(n-2p)rQ'-\mu''_{k,p} Q=0.
\end{eqnarray}
It is straightforward to check that a solution (at least $C^2$) on $[0,1]$ is given by $Q(r)=r^{k+p}$. The tangential harmonic extension of $\xi$ to the ball
is then:
$$
\widehat{\xi}=r^{k+p}\xi.
$$
Recall that, by definition, $T^{[p]}(\xi)=-i_{N}d\widehat{\xi}$. Now:
$$
d\widehat{\xi}=r^{k+p}\bar d\xi+(k+p)r^{k+p-1}dr\wedge\xi.
$$
On the boundary we have $r=1$ and $dr(N)=-1$, then:
$$
i_{N}d\widehat{\xi}=-(k+p)\xi
$$
so that
$$
T^{[p]}(\xi)=(k+p)\xi
$$
as asserted. 

\medskip

{\bf Proof of b)} Note that $\xi$ is co-closed and the associated Hodge-Laplace eigenvalue is $\mu''_{1,n}=0$.  Proceeding as in a), one finds that the tangential harmonic extension of $\xi$ is
$
\widehat\xi=r^{n+1}\xi
$
and therefore $T^{[n]}(\xi)=(n+1)\xi$.

\medskip

{\bf Proof of c)} We first observe that a Hodge-Laplace exact $p$-eigenform $\xi$ associated to $\mu'_{k,p}$ is the differential of a co-exact $(p-1)$-eigenform $\phi$, associated to the same eigenvalue $\mu''_{k,p-1}=\mu'_{k,p}$:
$$
\xi=\bar d\phi.
$$
From Lemma \ref{PolarEu}, we see that the $p$-form 
$Q\,\bar d\phi+P dr\wedge\phi$ is harmonic iff:
\begin{equation}\label{harm2}
\left\{
\begin{aligned}
\mu_{k,p-1}''(rP-2Q)&=r^3\left(P'+\dfrac{n-2p+2}{r}P\right)'\\
\mu_{k,p-1}'' Q-2rP&=r^2\left(Q''+\dfrac{n-2p}rQ'\right).
\end{aligned}\right.
\end{equation}

We observe that, by part a) of the proposition, the $(p-1)$-form $\phi$ is an eigenform for $T^{[p-1]}$ associated to $\nu_{k,p-1}''=k+p-1$ (note that if $p=1$ this fact follows directly from Theorem \ref{DNF}). 

\smallskip

Now consider the $p$-form:
$$
\widehat{\eta}=\alpha_{k,p} r^{k+p+1}\bar d\phi-\nu_{k,p-1}'' r^{k+p}dr\wedge\phi,
$$
where:
$$
\alpha_{k,p}=\dfrac{(\nu_{k,p-1}'')^2}{\mu_{k,p-1}''}=\dfrac{k+p-1}{n+k-p}.
$$
The $p$-form $\widehat{\eta}$ satisfies:
\begin{equation}\label{eta}
\left\{\begin{aligned}
&\Delta\widehat{\eta}=0\\
&\starred J\widehat{\eta}=\alpha_{k,p}\bar d\phi,\,\,i_{N}\widehat{\eta}=\nu_{k,p-1}''\phi,
\end{aligned}
\right.
\end{equation}
where the harmonicity follows from (\ref{harm2}) by taking:
$$
P(r)=-\nu_{k,p-1}'' r^{k+p},\quad Q(r)=\alpha_{k,p} r^{k+p+1}.
$$
Now let $\widehat{\phi}$ be the harmonic tangential extension of $\phi$, that is:
\begin{equation}\label{phi}
\left\{\begin{aligned}
&\Delta\widehat{\phi}=0\\
&\starred J\widehat{\phi}=\phi,\,\,i_{N}\widehat{\phi}=0.
\end{aligned}
\right.
\end{equation}
As $\phi$ is an eigenform for $T^{[p-1]}$ associated to $\nu_{k,p-1}''$, one has:
\begin{equation}\label{dphi}
\left\{\begin{aligned}
&\Delta d\widehat{\phi}=0\\
&\starred Jd\widehat{\phi}=\bar d\phi,\,\,i_{N}d\widehat{\phi}=-\nu_{k,p-1}''\phi.
\end{aligned}
\right.
\end{equation}
From \eqref{eta} and \eqref{dphi} one sees that the form $\widehat{\omega}=d\widehat{\phi}+\widehat{\eta}$ satisfies:
$$
\left\{\begin{aligned}
&\Delta \widehat{\omega}=0\\
&\starred J\widehat{\omega}=(\alpha_{k,p}+1)\bar d\phi,\,\,i_{N}\widehat{\omega}=0.
\end{aligned}
\right.
$$
From the definition of $T^{[p]}$, we have:
$$
T^{[p]}\big((\alpha_{k,p}+1)\bar d\phi\big)=-i_{N}d\widehat{\omega}.
$$
Now:
$$
d\widehat{\omega}=d\widehat{\eta}=\big(\nu_{k,p-1}''+\alpha_{k,p}(k+p+1)\big)r^{k+p}dr\wedge\bar d\phi
$$
and, restricted to the boundary:
$$
i_{N}d\widehat{\omega}=-\big(\nu_{k,p-1}''+\alpha_{k,p}(k+p+1)\big)\bar d\phi.
$$
This means that
$$
T^{[p]}((\alpha_{k,p}+1)\bar d\phi)=\big(\nu_{k,p-1}''+\alpha_{k,p}(k+p+1)\big)\bar d\phi
$$
and so $\xi=\bar d\phi$ is a Dirichlet-to-Neumann eigenform associated to the eigenvalue
$$
\dfrac{\nu_{k,p-1}''+\alpha_{k,p}(k+p+1)}{\alpha_{k,p}+1}=(k+p-1)\dfrac{n+2k+1}{n+2k-1}
$$
as asserted. 

\smallskip

This ends the proof of the proposition. 

\medskip


\subsection{On a variational problem for vector fields}\label{vectors}


Recall the variational problems defined in \eqref{dct} and \eqref{dcn}. Let $\mathbf{B}^3$ be the unit ball in $\real 3$. As a consequence of Corollary \ref{first} we have that, for any vector field $X$ on $\mathbf{B}^3$ tangent to $\bd \mathbf{B}^3=\sphere{2}$:
\begin{equation}\label{Dct}
\int_{\mathbf{B}^3}\Big(\abs{{\rm div}X}^2+\abs{{\rm curl}X}^2\Big)\geq \dfrac 53\int_{\sphere{2}}\abs {X}^2,
\end{equation}
because $\nu_{1,1}(\mathbf{B}^3)=\nu_{1,1}'(\mathbf{B}^3)=\frac 53$. Let us describe the minimizing vector fields. The eigenspace associated to $\nu'_{1,1}$ is $3$-dimensional, spanned by $\bar d\phi$, where $\phi$ is a linear function on $\real 3$. Hence the vector field $X$ attains equality in \eqref{Dct} iff it is dual to the harmonic tangential extension of $\bar d\phi$. Take for example $\phi=x_1$. As a consequence of the calculation done in the previous section, we see that the harmonic tangential extension of $\bar dx_1$ to the unit ball in $\real 3$ is, in rectangular coordinates:
$$
\widehat\xi=(2-2x_1^2+x_2^2+x_3^2)dx_1-3x_1x_2dx_2-3x_1x_3dx_3.
$$
Note that $\widehat\xi$ is a polynomial (not homogeneous) $1$-form. The dual vector field
$$
X=(2-2x_1^2+x_2^2+x_3^2)\dfrac{\partial}{\partial x_1}-3x_1x_2\dfrac{\partial}{\partial x_2}-3x_1x_3\dfrac{\partial}{\partial x_3}
$$
is then a minimizer for the variational problem \eqref{Dct}.

\smallskip

On the other hand, as $\nu_{1,2}(\mathbf{B}^3)=\nu_{1,2}''(\mathbf{B}^3)=3$, one sees that, for any vector field normal to $\sphere{2}$ we have:
\begin{equation}\label{Dcn}
\int_{\mathbf{B}^3}\Big(\abs{{\rm div}X}^2+\abs{{\rm curl}X}^2\Big)\geq 3\int_{\sphere{2}}\abs {X}^2.
\end{equation}
Now $\nu_{1,2}''$ is simple and is spanned by the volume form $v$ of $\sphere 2$. The tangential harmonic extension of $v$ to the unit ball is
$$
\widehat v=x_1 dx_2\wedge dx_3+x_2dx_3\wedge dx_1+ x_3dx_1\wedge dx_2.
$$ 
The space of minimizing vector fields for \eqref{Dcn} is then one-dimensional, spanned by the dual of $\star\widehat v$, that is, by the radial vector field
$$
Y=x_1\dfrac{\partial}{\partial x_1}+x_2\dfrac{\partial}{\partial x_2}+
x_3\dfrac{\partial}{\partial x_3}.
$$


\subsection{Proof of Lemma \ref{PolarEu}}\label{polar}


Assume that $(M^{n+1},g)$ is a rotationally symmetric manifold, that is
$$
M^{n+1}=[0,\infty)\times\sphere n,
$$
endowed with the Riemannian metric
$$
g=dr^2\oplus\theta(r)^2ds^2_n,
$$
where $ds^2_n$ is the canonical metric on $\sphere n$ and $\theta$ is a smooth positive function on $(0,\infty)$. Of course, one gets the space form of curvature $-1,0,1$ when $\theta=\sinh r, r, \sin r$, respectively. To prove Lemma \ref{PolarEu} we will take $\theta(r)=r$. In this setting, any $p$-form can be split into its  tangential and normal components:
$$
\omega=\omega_1+dr\wedge\omega_2,
$$
where $\omega_1$ and $\omega_2$ are forms of degrees $p$ and $p-1$, respectively. We assume that we can separate the variables, that is:
$$
\omega_1(r,x)=Q(r)\xi(x),\quad \omega_2(r,x)=P(r)\eta(x),
$$
where $\xi\in\Lambda^p(\sphere n)$, $\eta\in\Lambda^{p-1}(\sphere n)$ and $P,Q$ are radial functions so that
$$
\omega=Q\xi+dr\wedge (P\eta).
$$


\subsubsection{A suitable orthonormal frame} 


Fix a point $(r,x)\in M$ with $r\in\reals$ and $x\in\sphere n$ and let $(\bar E_1,\dots,\bar E_n)$ be an orthonormal frame on the sphere which is geodesic at $x$ for the canonical metric $ds^2_n$. If we set $Z=\partial/\partial r$ and $E_j=\theta^{-1}\bar E_j$ for $1\leq j\leq n$, it is obvious that the frame $(Z,E_1,\dots, E_n)$ is $g$-orthonormal. From the Koszul formula, we compute:
\begin{eqnarray}\label{NF}
\nabla_{E_j}E_k=-\delta_{jk}\dfrac{\theta'}{\theta} Z,\,\,\nabla_{E_j}Z=\prin E_j,\,\,\nabla_{Z}E_j=0,\,\,\nabla_ZZ=0
\end{eqnarray}
where $\nabla$ denotes the Levi-Civita connection on $(M,g)$. In particular, we have $\nabla_{E_j}E_k=0$ if $j\neq k$ and
\begin{eqnarray}\label{NF1}
\nabla_{E_j}E_j=-\prin Z, \quad \sum_j\nabla_{E_j}E_j=-n\prin Z.
\end{eqnarray}


\subsubsection{Computing the divergence and the differential}


In this part, we first compute the divergence of a $p$-form using the orthonormal frame of the previous section. If $\omega$ is a $p$-form on $M$, then from the definition of the divergence and using the relations (\ref{NF}) and (\ref{NF1}), we have:
$$
\begin{aligned}
\delta\omega(E_1,\dots,E_{p-1}) = & -\sum_{j=1}^nE_j\cdot\omega(E_j,E_1,\dots,E_{p-1})-Z\cdot\omega(Z,E_1,\dots,E_{p-1})\\
&-(n-p+1)\prin \omega(Z,E_1,\dots,E_{p-1})
\end{aligned}
$$
and 
$$
\delta\omega(Z,E_1,\dots,E_{p-2})=\sum_{j=1}^n E_j\cdot\omega(Z,E_j,E_1,\dots,E_{p-2}).
$$
Let $\phi$ be a $p$-form on $\sphere n$. As the frame $(\bar E_1,\dots,\bar E_n)$ of $\Sigma$ is geodesic at the point $x\in\Sigma$ and the function $\phi(\bar E_{j_1},\dots,\bar E_{j_{p}})$ is constant in the $r$-direction for any choice of $j_1,\dots,j_{p}$, we have:
$$
\twosystem
{\sum_{j=1}^nE_j\cdot\phi(E_j,E_1,\dots,E_{p-1})=-\dfrac{1}{\theta^2}\bar\delta\phi(E_1,\dots,E_{p-1})}
{Z\cdot \phi(E_1,\dots,E_p)=-p\dfrac{\theta'}{\theta}\phi(E_1,\dots,E_p)}
$$

Since the vectors $E_1,\dots,E_{p}$ can be replaced by any set of tangential vectors in the chosen frame, a straightforward calculation using the above equations shows that,  
for $\eta\in\Lambda^{p-1}(\sphere n)$ and $\xi\in\Lambda^p(\sphere n)$:
\begin{equation}\label{codiff}
\twosystem
{\delta(dr\wedge (P\eta))=dr\wedge(-\dfrac{P}{\theta^2}\bar\delta\eta)-
\left[ P'+(n-2p+2)\prin P\right]\eta}
{\delta(Q\xi)=\dfrac{Q}{\theta^2}\bar\delta\xi.}
\end{equation}
For the differential, it is clear that we directly have:
\begin{equation}\label{diff}
\twosystem
{d(dr\wedge (P\eta))=-dr\wedge(P\bar d\eta)}
{d(Q\xi)=dr\wedge(Q'\xi)+Q\bar d\xi.}
\end{equation}
At this point,  using (\ref{codiff}) and (\ref{diff}), one proves, after some standard work:
\begin{lemme}\label{pol} 
Let $\omega=Q\xi+Pdr\wedge \eta$ where $\eta\in\Lambda^{p-1}(\sphere n)$ and $\xi\in\Lambda^p(\sphere n)$. Then:
$$
\Delta\omega=\omega_1+dr\wedge\omega_2,
$$
where:
$$
\twosystem
{\omega_1=\dfrac{Q}{\theta^2}\bar\Delta\xi-\big(Q''+(n-2p)\prin Q'\big)\xi-2\prin P\bar d\eta}
{\omega_2=\dfrac{P}{\theta^2}\bar\Delta\eta-\big(P'+(n-2p+2)\prin P\big)'\eta-2Q
\dfrac{\theta'}{\theta^3}\bar\delta\xi.}
$$
\end{lemme}

To prove Lemma \ref{PolarEu}, we now set $\theta(r)=r$. 

\begin{rem} \rm Let $B_R$ be the geodesic ball of radius $R$ centered at the pole of a rotationally symmetric manifold, that is $B_R=\{(r,x)\in [0,\infty)\times \sphere n: r\leq R\}$.  One can construct eigenforms of the Dirichlet-to-Neumann operator on $B_R$ by solving ordinary differential equations. 

\smallskip

For example, one sees from the above expression of the Laplacian that the harmonic tangential extension of a co-closed eigenform $\xi\in\Lambda^p(\sphere n)$ for the Hodge Laplacian on the sphere associated with $\mu_{k,p}''$ is given by $\hat\xi=Q\xi$, where $Q(r)$ is the smooth function satisfying:
$$
\twosystem
{{\frac{1}{\theta^{n-2p}}\Big(\theta^{n-2p}Q'\Big)'=\frac{\mu''_{k,p} Q}{\theta^2}},}
{Q(R)=1.}
$$
It follows directly from the definition that then
$$
\nu''_{k,p}=Q'(R)
$$
is an eigenvalue of $T^{[p]}$.

Now take $p=n$. Then,  the only possible value of $\mu_{k,n}''$ is zero, corresponding to the volume form of $\sphere n$, which is parallel hence harmonic. In that case the previous equation reduces to 
$$
\Big(\theta^{-n}Q'\Big)'=0,
$$
hence, up to multiples, $Q'(r)=\theta^n(r)$. Then:
\begin{equation}\label{rotball}
\nu_{1,n}''(B_R)=\dfrac{\theta^n(R)}{\int_0^R\theta^n(r)dr}=\dfrac{\Vol(\bd B_R)}{\Vol(B_R)}
\end{equation}
is an eigenvalue for the Dirichlet-to-Neumann operator on $n$-forms.

\smallskip

 In \cite{R-S} we called a domain $\Omega$ {\it harmonic} if the mean-exit time function $E$ of $\Omega$, solution of the problem \eqref{met},  has constant normal derivative on $\Sigma$: we then observed that for a harmonic domain ${\rm Vol}(\Sigma)/{\rm Vol}(\Omega)$ is always an eigenvalue of $T^{[n]}$. A geodesic ball $B_R$, being rotationally invariant, is a harmonic domain because the function $E$ is also rotationally invariant.  Therefore the above result \eqref{rotball} is not a surprise. 
 
 \end{rem}


\section{Proof of Theorem \ref{newupper}}


Let $\Omega$ be a Riemannian domain with smooth boundary $\Sigma$ and inner unit normal vector field $N$. Consider the shape operator of $\Sigma$,  defined by $S(X)=-\nabla_XN$ for all $X\in T\Sigma$. Then $S$ can be extended to a self-adjoint operator $S^{[p]}$ acting on $p$-forms of $\Sigma$ by the rule:
$$
S^{[p]}(\omega)(X_1,\dots,X_p)=\sum_{j=1}^p\omega(X_1,\dots,S(X_j),\dots,X_p),
$$
for all $\omega\in\Lambda^{p}(\Sigma)$ and for all vectors $X_1,\dots,X_p\in T\Sigma$.

\smallskip

Let $(e_1,\dots,e_n)$ be an orthonormal basis of principal directions, so that 
$S(e_{j})=\eta_je_j$ for all $j$, where $\eta_1, \dots,\eta_n$ are the principal curvatures of $\Sigma$. Then, for any multi-index $1\leq j_1<\dots<j_p\leq n$ one has
$$
S^{[p]}(\omega)(e_{j_1},\dots,e_{j_p})=(\eta_{j_1}+\dots+\eta_{j_p})\omega(e_{j_1},\dots,e_{j_p}).
$$
 In particular, if $\omega$ is an $n$-form:
$$
S^{[n]}(\omega)=nH\omega,
$$
where $H$ is the mean curvature of $\Sigma$.

\smallskip

For later use we observe that, if $\xi$ is a $p$-form on $\Omega$ and 
$L_N\xi=di_N\xi+i_Nd\xi$ is its Lie derivative along $N$, then it follows directly from the definitions that, on $\Sigma$, one has (see for example Lemma $18$ in \cite{raulotsavo}):
\begin{equation}\label{lie}
\starred J(L_N\xi)=\starred J(\nabla_N\xi)-S^{[p]}(\starred J\xi).
\end{equation}

The proof of Theorem \ref{newupper} is based on the following estimates.
\begin{prop}\label{prop} 
Let $\xi$ be an exact parallel $p$-form on $\Omega$, with $p=1,\dots,n$. If $p=1$, then
\begin{eqnarray}\label{newcarf}
\nu_{2,0}(\Omega)\int_{\Omega}\abs{\xi}^2\leq\int_{\Sigma}\abs{i_N\xi}^2.
\end{eqnarray}
and if $p=2,\dots,n$:
\begin{eqnarray}\label{newcarp}
\nu_{1,p-1}(\Omega)\int_{\Omega}\abs{\xi}^2\leq\int_{\Sigma}\abs{i_N\xi}^2.
\end{eqnarray}
Moreover, if equality holds in \eqref{newcarf} or \eqref{newcarp}, then:
$$
S^{[p]}(\starred J\xi)=\nu\starred J\xi,
$$
where $\nu=\nu_{2,0}(\Omega)$ when $p=1$ and $\nu=\nu_{1,p-1}(\Omega)$ when $p\geq 2$.
\end{prop}

\begin{proof} 
Inequalities \eqref{newcarf} and \eqref{newcarp} were proved in \cite{R-S} and hold, more generally, only assuming that $\xi$ is exact and satisfies $d\xi=\delta\xi=0$ on $\Omega$. Let us recall the main argument. As $\xi$ is exact, there exists a unique co-exact $(p-1)$-form $\theta$, called the {\it canonical primitive} of $\xi$, satisfying:
$$
\twosystem
{d\theta=\xi\quad\text{on $\Omega$},}
{i_N\theta=0 \quad\text{on $\Sigma$}.}
$$
When $p=1$, we take $\theta$ as the unique primitive of $\xi$ such that
$\int_{\Sigma}\theta=0$.

We now take $\theta$ as test $(p-1)$-form in the min-max principles (\ref{car02}) and (\ref{var}) and the inequalities (\ref{newcarf}) and (\ref{newcarp}) follow after some easy work (see \cite{R-S}). Now assume that equality holds: then $\theta$ is an eigenform associated to $\nu$, so that $i_Nd\theta=-\nu \starred J\theta$. This means:
$$
i_N\xi=-\nu \starred J\theta
$$
on $\Sigma$. Differentiating on $\Sigma$ one gets
$
d^{\Sigma}i_N\xi=-\nu\starred J\xi$, and in turn, as $\xi$ is closed:
$$
\starred J(L_N\xi)=-\nu\starred J\xi.
$$
As $\xi$ is parallel, \eqref{lie} gives $\starred J(L_N\xi)=-S^{[p]}(\starred J\xi)$ and then the assertion follows. 
\end{proof}

\medskip

We can now give the proof of Theorem \ref{newupper}. We can assume that $p\leq n-1$: in fact, the assertion for $p=n$ is a direct consequence of Theorem 5 and Corollary 6 in \cite{R-S}.

Let $\Cal P$ be the family of unit length parallel vector fields on $\real {n+1}$: then $\cal P$ is naturally identified with $\sphere n$.  For $V_1,\dots,V_p\in\Cal P$ consider the parallel $p$-form
$$
\xi=\starred{V_1}\wedge\dots\wedge\starred{V_p}
$$
where $\starred{V_j}$ denotes the dual $1$-form of $V_j$. Note that $\starred{V_j}$ is the differential of a linear function, so $\xi$ is also exact. Then, denoting by $\nu$ the eigenvalue as in Proposition \ref{prop}, we have from (\ref{newcarf}) and (\ref{newcarp})
\begin{equation}\label{in}
\nu\int_{\Omega}\abs{\xi}^2\leq\int_{\Sigma}\abs{i_N\xi}^2.
\end{equation}
We wish to integrate this inequality over all $(V_1,\dots,V_p)\in(\sphere{n})^p$. To that end, we introduce on ${\cal P}=\sphere n$ the measure:
$$
d\mu=\dfrac{n+1}{\Vol(\sphere n)}ds^2_n,
$$
where $ds^2_n$ is the canonical measure on $\sphere n$. The normalization is chosen so that, at each point $x\in\real{n+1}$, and for all tangent vectors $X,Y$ at $x$ one has:
\begin{equation}\label{parallel}
\int_{\sphere n}\scal{V}{X}\scal{V}{Y}d\mu(V)=\scal{X}{Y}.
\end{equation}
A straightforward calculation using \eqref{parallel} (explicitly done in Lemma 2.1 of \cite{savo}) then gives, at each point of $\Omega$ (respectively, of $\Sigma$):
$$
\twoarray
{\int_{(\sphere{n})^p}\abs{\starred{V_1}\wedge\dots\wedge\starred{V_p}}^2d\mu(V_1)\dots d\mu(V_p)=p!\binom{n+1}{p},}
{\int_{(\sphere{n})^p}\abs{i_N(\starred{V_1}\wedge\dots\wedge\starred{V_p})}^2d\mu(V_1)\dots d\mu(V_p)=p!\binom{n}{p-1}.}
$$
We now integrate both sides of \eqref{in} with respect to $(V_1,\dots,V_p)\in(\sphere{n})^p$. If $p=1$ we get:
$$
\nu_{2,0}(\Omega)\leq \dfrac{1}{n+1}\isoper,
$$
and if $p=2,\dots,n$ we get:
$$
\nu_{1,p-1}(\Omega)\leq \dfrac{p}{n+1}\isoper,
$$
which, after replacing $p$ by $p+1$, are the inequalities stated in  the theorem.

\medskip

We now discuss the equality case. If equality holds then, from Proposition \ref{prop}, we see that  
\begin{equation}\label{equality}
S^{[p]}(\starred J\xi)=\nu\starred J\xi
\end{equation}
 for all such $\xi=\starred{V_1}\wedge\dots\wedge\starred{V_p}$. Fix a point 
$x\in\Sigma$ and an orthonormal frame $(e_1,\dots,e_n)$ of principal directions at $x$. Fix a multi-index $j_1<\dots<j_p$ and choose:
$$
V_1=e_{j_1},\dots, V_p=e_{j_p}.
$$
At $x$, one has $\starred J\xi(e_{j_1},\dots,e_{j_p})=1$ and then, by the definition of $S^{[p]}$:
$$
S^{[p]}(\starred J\xi)(e_{j_1},\dots,e_{j_p})=\eta_{j_1}(x)+\dots+\eta_{j_p}(x),
$$
the corresponding $p$-curvature at $x$. From \eqref{equality} we then get:
$$
\nu=\eta_{j_1}(x)+\dots+\eta_{j_p}(x).
$$
This holds for all multi-indices $j_1<\dots<j_p$ and for all $x\in\Sigma$: hence, all $p$-curvatures are constant on $\Sigma$,  and equal to $\nu$. If $p<n$, this immediately implies that $\Sigma$ is totally umbilical, hence it is a sphere. If $p=n$, we have that the mean curvature of $\Sigma$ is constant; by the well-known Alexandrov theorem $\Sigma$ is, again, a sphere. 

Finally, from Theorem \ref{DNS}, we have that if $\Omega$ is a ball, then equality holds for $\nu_{2,0}$ and for $\nu_{1,p-1}$ provided that $p-1\geq (n+1)/2$. The proof is complete.



\vspace{0.8cm}     
Authors addresses:     
\nopagebreak     
\vspace{5mm}\\     
\parskip0ex     
\vtop{\hsize=6cm\noindent\obeylines}     
\vtop{     
\hsize=8cm\noindent     
\obeylines     
Simon Raulot
Laboratoire de Math\'ematiques R. Salem
UMR $6085$ CNRS-Universit\'e de Rouen
Avenue de l'Universit\'e, BP.$12$
Technop\^ole du Madrillet
$76801$ Saint-\'Etienne-du-Rouvray, France}     
     
\vspace{0.5cm}     
     
E-Mail:     
{\tt simon.raulot@univ-rouen.fr }  

\vtop{\hsize=6cm\noindent\obeylines}     
\vtop{     
\hsize=9cm\noindent     
\obeylines     
Alessandro Savo
Dipartimento SBAI, Sezione di Matematica 
Sapienza Universit\`a di Roma
Via Antonio Scarpa 16
 00161 Roma, Italy         
}     
     
\vspace{0.5cm}     
     
E-Mail:     
{\tt alessandro.savo@sbai.uniroma1.it  } 


\end{document}